\documentclass[twocolumn,10pt]{IEEEtran}

\usepackage{graphicx,subfigure,booktabs,pgfplots,bm}
\usepackage{epsfig}
\usepackage{amsmath}
\usepackage{amssymb,amsthm}
\usepackage{algorithm,algorithmic}
\usepackage[T1]{fontenc}
\usepackage{booktabs}
\usepackage{cite}
\usepackage{stackengine}
\usepackage{enumitem}

\newtheorem{prop}{Proposition}

\newtheorem{lemma}{Lemma}

\newtheorem{assumption}{Assumption}
\newtheorem{corol}{Corollary}

\title{Shortest Dubins Path to a Circle}

\author{Satyanarayana G. Manyam\thanks{Corresponding author. Email: msngupta@gmail.com} \and David Casbeer \and Alexander L. Von Moll \and Zachariah Fuchs
}

\begin{document}

\maketitle    

\begin{abstract}
{\it The Dubins path problem had enormous applications in path planning for autonomous vehicles. In this paper, we consider a generalization of the Dubins path planning problem, which is to find a shortest Dubins path that starts from a given initial position and heading, and ends on a given target circle with the heading in the tangential direction. This problem has direct applications in Dubins neighborhood traveling salesman problem, obstacle avoidance Dubins path planning problem etc. We characterize the length of the four CSC paths as a function of angular position on the target circle, and derive the conditions which to find the shortest Dubins path to the target circle. }
\end{abstract}


\section{INTRODUCTION}
Dubins path planning problem is widely used for path planning and trajectory planning for unmanned aerial vehicles with finite yaw rate. For the vehicles with finite yaw rate, it is natural to use Dubins paths to generate flyable trajectories that satisfy the curvature constraints. Given an initial and final points in a plane, and a direction at these two points, a Dubins path gives the shortest path between these points that satisfy the minimum turn radius constraints. 

There are several results in the literature related to the Dubins paths \cite{bui1994accessibility, bui1994shortest, yang2002optimal, wong2004uav, manyam2017tightly, manyam2018tightly} . In \cite{bui1994accessibility, bui1994shortest}, the analysis of the accessibility regions of Dubins paths is done and the Dubins synthesis problem are presented. The three points Dubins problem which is a generalization of the Dubins path planning problem is presented in \cite{yang2002optimal, wong2004uav}; here, an initial and final configuration is prescribed along with a third point. The curvature constrained path between initial and final points should pass through the given third point. The curvature constrained path planning problem in the presence of wind is addressed in \cite{mcgee2005optimal,techy2009minimum}. The problem of finding shortest curvature constrained path in the presence of obstacles is presented in \cite{boissonnat1996polynomial, macharet2009generation, agarwal1995motion, maini2016path}. Another generalization of the Dubins interval problem is presented in \cite{manyam2017tightly, manyam2018tightly}; it gives the algorithm to solve the shortest curvature constrained path between two points, where the heading is restricted to prescribed intervals at the initial and the final points. This generalization helped in improving the lower bounds for Dubins traveling salesman problem significantly.

In this paper, we propose another generalization of the Dubins path planning problem: Given an initial location and direction, a fixed target circle and rotational direction, find the shortest curvature constrained path from initial configuration to a point on the circle, such that the final heading of the path is tangent to the circle in the prescribed direction. This fundamental problems has significant applications such as the obstacle avoidance path planning, neighborhood Dubins traveling salesman problem \cite{isaacs2011algorithms, macharet2012evolutionary, hespanha}. This problem also arises when finding shortest path for a pursuit evader problem where evader is following a cyclical path. Thus the presented shortest Dubins path to a target circle has significant applications in several path planning problems. 

In \cite{Dubins1957}, Dubins states that the shortest path consists of at most three segments, where each segment could be a circular arc or a straight line. If we represent the circular arc with C and, the straight line with S, the shortest path could be of the type CSC or CCC. Let L and R represent counter-clockwise and clockwise circular arcs respectively; the shortest path should be one of the following six combinations: LSL, RSL, RSR, LSR, LRL, and RLR. When the distance between the points in greater than four times the minimum turn radius, the shortest path could be one of the four combinations of CSC paths \cite{bui1994shortest, goaoc2013bounded}. 

Clearly, the shortest path from an initial configuration to a point on the circle also could only consists of one of these six combinations, or the degenerate cases there of. If one could find the shortest possible path of each of these cases, that starting from the initial configuration and ends on the circular with the final direction tangential and at given orientation, then the minimum of all the six cases gives the shortest Dubins path to the circle. In this paper, we address the four cases of the CSC paths, \textit{i.e} LSL, RSL, RSR, LSR. Under the assumption that the straight line distance between the initial position and any point on the circle is greater than four times the minimum turn radius, it is sufficient to just look at these four cases, however one needs to analyze all the six cases to find the optimal path for the general case. As a first step towards solving the general case, we address the problem with the assumption on distances. 

\subsection{Problem Statement}
Given an initial configuration ($x$, $y$ coordinates and the heading direction) $(x_i,y_i,\theta_i)$, the target circle, and the rotational direction (clockwise/counter-clockwise), find the shortest path, subject to the minimum turn radius constraints, from the initial configuration to a point on the target circle with final heading direction tangential to the circle in the specified rotational direction. The problem setup and feasible paths with final heading tangential to the circle in clockwise and counter-clockwise directions are shown in the Figs. \ref{fig:pathCCW} and \ref{fig:pathCW} respectively.


\begin{figure}[htpb]
\begin{center}
\subfigure[Final tangential direction is clockwise to the target circle]{\includegraphics[width=3in]{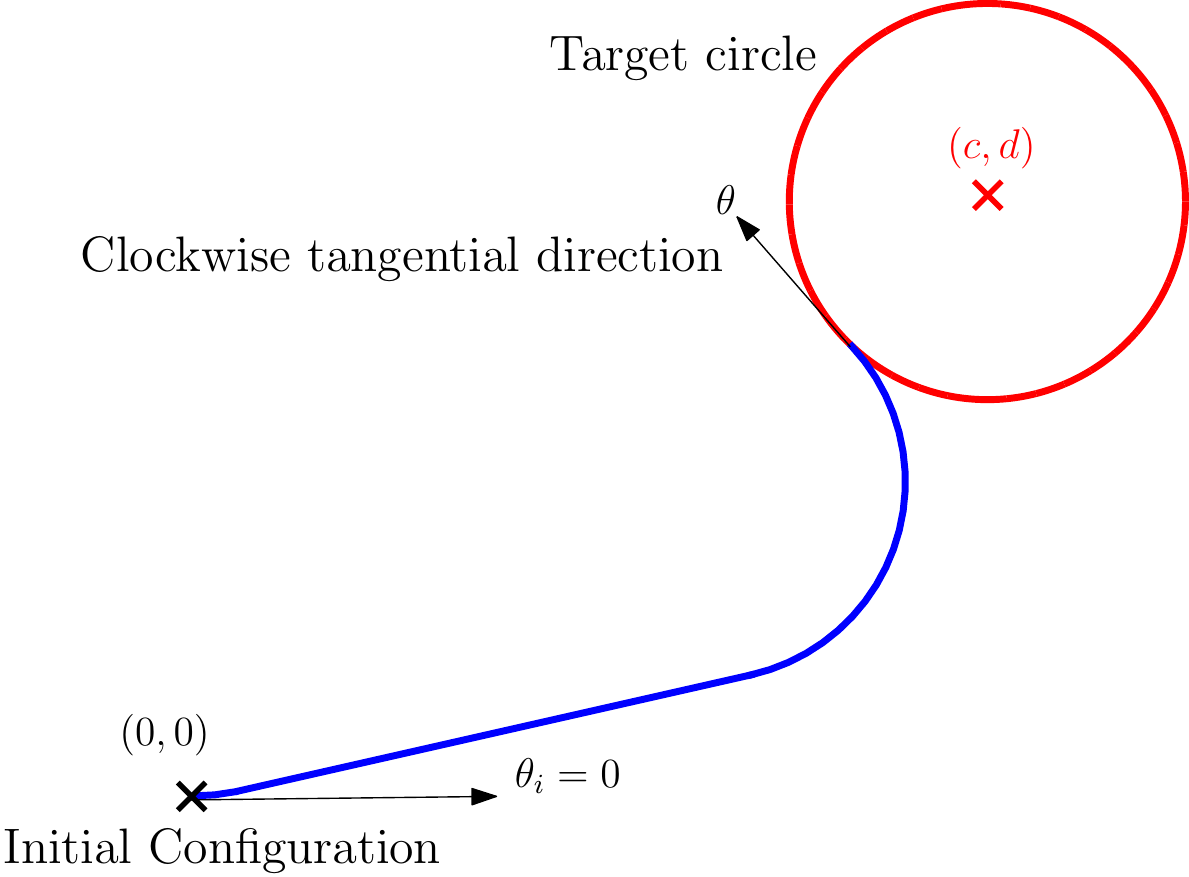}\label{fig:pathCCW}}
\subfigure[Final tangential direction is counter-clockwise to the target circle]{\includegraphics[width=3in]{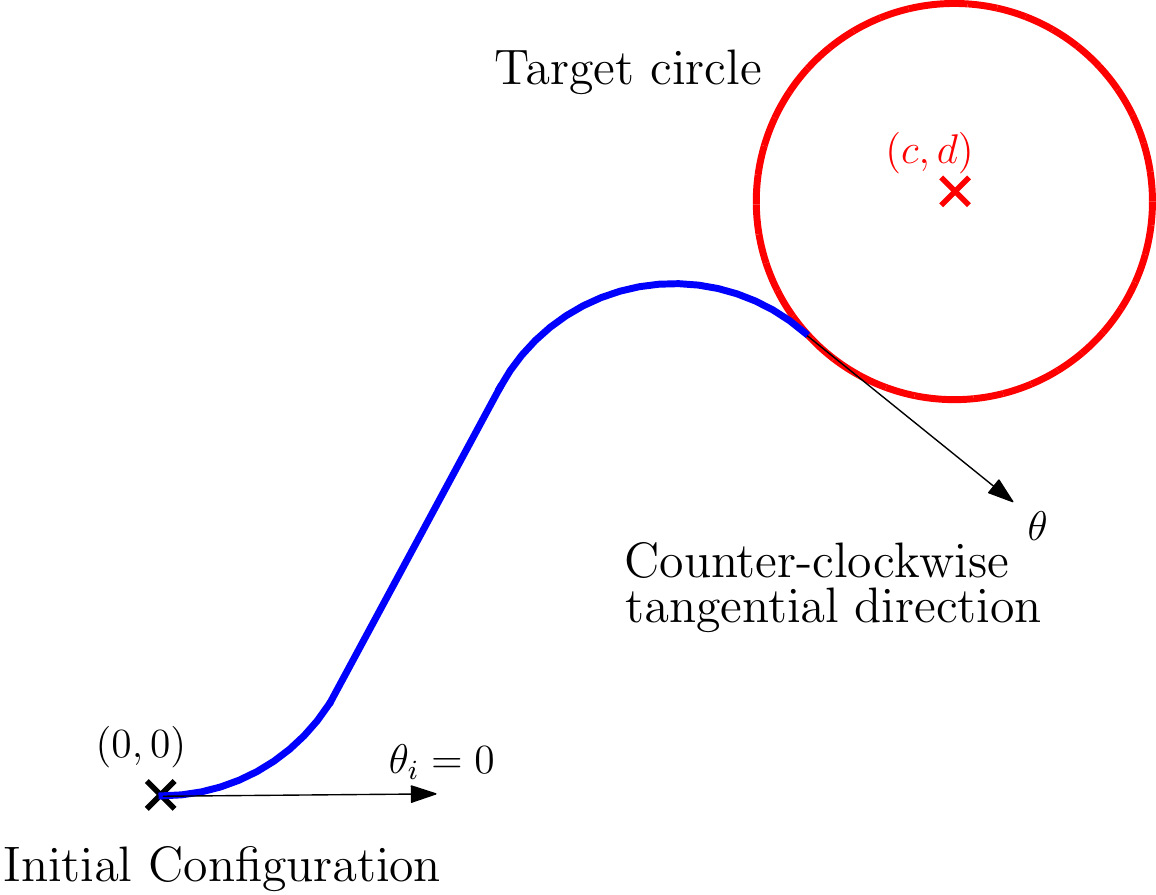}\label{fig:pathCW}}
\end{center}
\caption{Feasible paths to a circle with final direction tangential to the target circle}
\label{figure_ASME} 
\end{figure}

\subsection*{Notations:}

\begin{table}[h!]
\renewcommand \caption [2][]{}
\caption{}
\begin{center}
\label{tab:ressum1}
\begin{tabular}{rl}
$C_1/C_2:$ & First or second arc/circle of the CSC path \\
$C_3:$ & Target circle \\
$r$: &  Minimum turn radius and radius of the target circle\\
$\theta:$&  Final heading of the Dubins path\\
$\alpha:$&  Angular position of the final point on the target circle \\
$\phi_1/ \phi_2:$&  Angle subtended by the first/second arc of the CSC path\\
$L_S:$&Length of the straight line/middle segment of the CSC path
\end{tabular}
\end{center}
\end{table}


\section{Main Result} \label{sec:main}

\begin{assumption} \label{assum:4r}
The distance between the initial and final position is always greater than four times the minimum turn radius ($r$).
\end{assumption}
\begin{assumption}\label{assum:equalr}
The radius of the target circle $C_3$ and the minimum turn radius of the vehicle are equal.
\end{assumption}

To find the shortest Dubins paths, one can aim to find the shortest paths of six classes of Dubins paths, and the one with minimum length among these six paths will be the shortest Dubins path to the circle. However, we restrict our analysis in this paper to the four types of the CSC paths, which is sufficient to find the optimal path under the Assumption \ref{assum:4r}. Note that, the analysis done in this paper is not just restricted to only the instances that satisfy Assumption \ref{assum:4r}, it also applies to any type of CSC paths when they exist.  Furthermore, it is clear from the symmetry and the equivalency between dubins paths \cite{shkel2001dubins}, we need to analyze only two types of paths among the four CSC paths. The RSR and LSR paths are symmetrical to the LSL and RSL paths respectively, and therefore the results for LSL and RSL paths are directly applicable to the other two paths. 

The heading direction at the final position on the circle is tangential to the circle, and this direction could be either clockwise or counter-clockwise to the target circle. Depending on the rotational direction chosen on the target circle, the shortest Dubins paths occur at different positions on the target circle. We analyze the LSL and RSL paths for these two cases separately in the subsections \ref{subsec:cw} and \ref{subsec:ccw}.

\subsection{Clockwise tangential direction}\label{subsec:cw}

\begin{prop}\label{thm:csccw}
The maximum and minimum of the the LSL and the RSL paths with respect to $\alpha$ occur when the direction of the straight line segment passes through the center of the target circle.
\end{prop}
Given the final heading direction is clockwise tangent to the target circle, we have $\theta = \alpha - \frac{\pi}{2}$. Using this relation, we prove the Proposition \ref{thm:csccw} in the following subsections. The plot of the length of the paths against the angular position on the target circle are shown in the Fig. \ref{fig:cscalcw}. Though we see only one discontinuity for each plot, there could be potentially two positions where the length is discontinuous.

\begin{figure}[h]
\begin{center}
\includegraphics[width=2.75in]{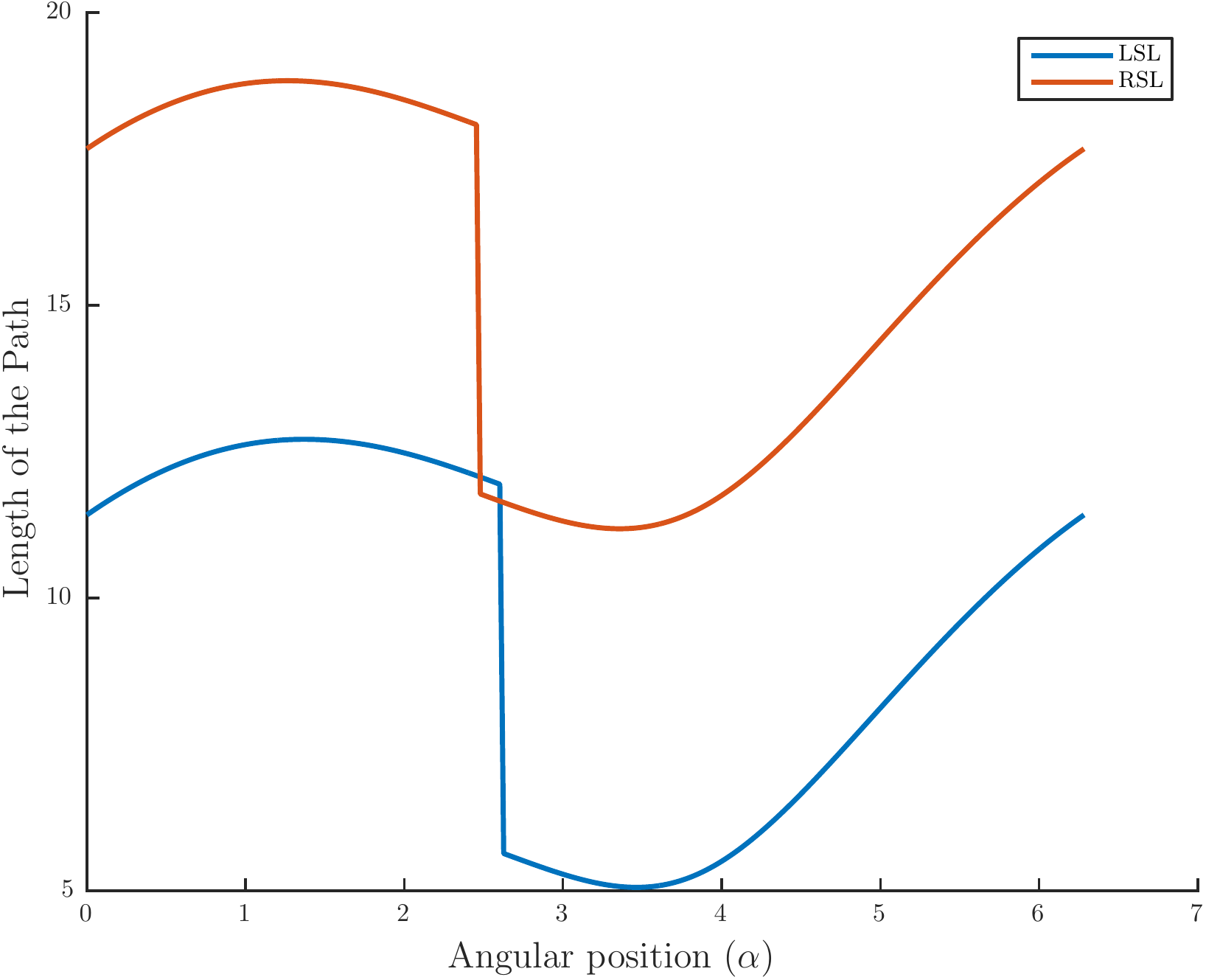}\end{center}
\caption{The length of the LSL and RSL paths vs the angular position on target circle. Final headings are clockwise tangents to the target circle.}
\label{fig:cscalcw} 
\end{figure}

\subsubsection{LSL Paths}
\begin{figure}[h]
\begin{center}
{\includegraphics[width=3.25in]{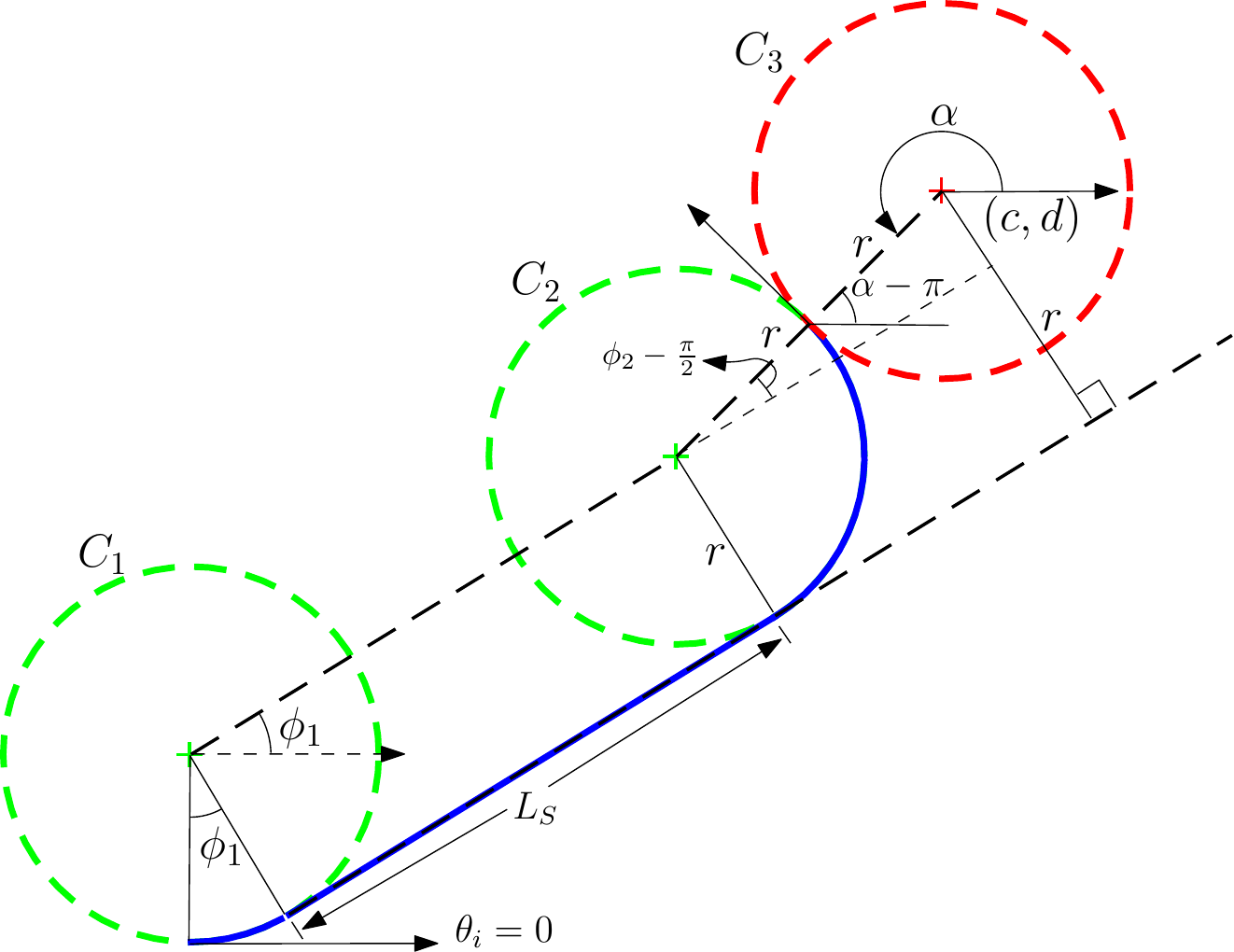}}
\end{center}
\caption{LSL Path}
\label{fig:pathlsl} 
\end{figure}
Without loss of generality we assume that the initial position is at the origin and the initial heading is towards the positive $x$-axis, i.e. the initial heading is at zero degrees with respect to $x$-axis as shown in the Fig. \ref{fig:pathlsl}. Let $(c,d)$ be the centre of the target circle and $r$ be its radius. We will express the length of the $LSL$ path for an arbitrary position on the target circle as a function of the angular position, $\alpha$.

Let $(x,y)$ be the coordinates of the final position of the Dubins path, and $\theta$ be the final heading direction. Using elementary geometry, one can see the centre of the second circle $C_2$ is $(x-r \sin \theta, y+ r\cos \theta)$. The length of the $LSL$ path is given by the sum of the three segments, the first circular arc, the straight line and the second circular arc. Let $\phi_1$ and $\phi_2$ be the angle subtended by the first and second arc respectively, and let $L_S$ be the length of the straight line (second segment of the CSC path). The length $L_S$ is equal to the distance between the centers of the circles $C_1$ and $C_2$. The length of the LSL path is given by the sum of these three segments $L_{LSL} =  L_S + r(\phi_1+\phi_2)$. Using geometry, one can derive $L_S$, $\phi_1$ and $\phi_2$, and are given as the following:

\begin{flalign}
L_S &= \sqrt{(x-r\sin \theta)^2 +(y + r\cos \theta -r)^2}, \label{eqn:lslls}\\
\phi_1 &= \mod\left(\mbox{atan2}\left(\frac{y+r\cos \theta -r}{x - r\sin \theta}\right),2\pi\right), \label{eqn:lslphi1}\\
\phi_2 &= \mod(\theta - \phi_1, 2 \pi). \label{eqn:lslphi2}
\end{flalign}

For this case, as the final heading is in the clockwise tangential direction at the target circle $\theta  = \alpha - \frac{\pi}{2}$, and by substituting $(x,y) = (c+r \cos\alpha, d+r \sin\alpha )$ in the equations (\ref{eqn:lslls} - \ref{eqn:lslphi1}) gives the following:
\begin{flalign}
L_S &= \sqrt{(c+2 r \cos \alpha)^2 +(d + 2r\sin \alpha-r)^2}, \label{eqn:lsllscw}\\
\phi_1 &= \mod\left(\mbox{atan2}\left(\frac{d + 2r\sin \alpha-r}{c+2 r\cos \alpha}\right),2\pi\right). \label{eqn:lslphi1cw}
\end{flalign}
One can find the maximum and minimum of the LSL paths by solving the following
\begin{equation}
 \frac{d}{d \alpha} \left( L_s+r(\phi_1+\phi_2) \right)= 0. \label{eqn:ddallsl}
 \end{equation}
 
 \begin{lemma}\label{lem:lslmin}
 The maximum or minimum of the length of the LSL path with respect to the angular position on the target circle ($\alpha$) occurs when $\phi_2=\frac{\pi}{3}$ or $\frac{5 \pi}{3}$, and at these two positions, the direction of the straight line of the LSL path (second segment) passes through the center of the target circle.
 \end{lemma}
\begin{proof}
Refer to the Appendix for the proof.
\end{proof}

\begin{corol}
The absolute value of the derivative of the length of the LSL path, $\frac{d}{d \alpha} L_{LSL}$ represents the perpendicular distance from the center of the target circle to the straight line segment of the LSL path.
\end{corol}
\begin{proof}
From Fig. \ref{fig:pathlsl}, the perpendicular distance from the center of the target circle to the straight line segment of the LSL path is $r+2rsin(\phi_2-\frac{\pi}{2})$, which is equal to $r-2rcos(\phi_2)$. 
\end{proof}
\subsubsection{RSL Paths}
\begin{figure}[htpb]
\begin{center}
{\includegraphics[width=3.25in]{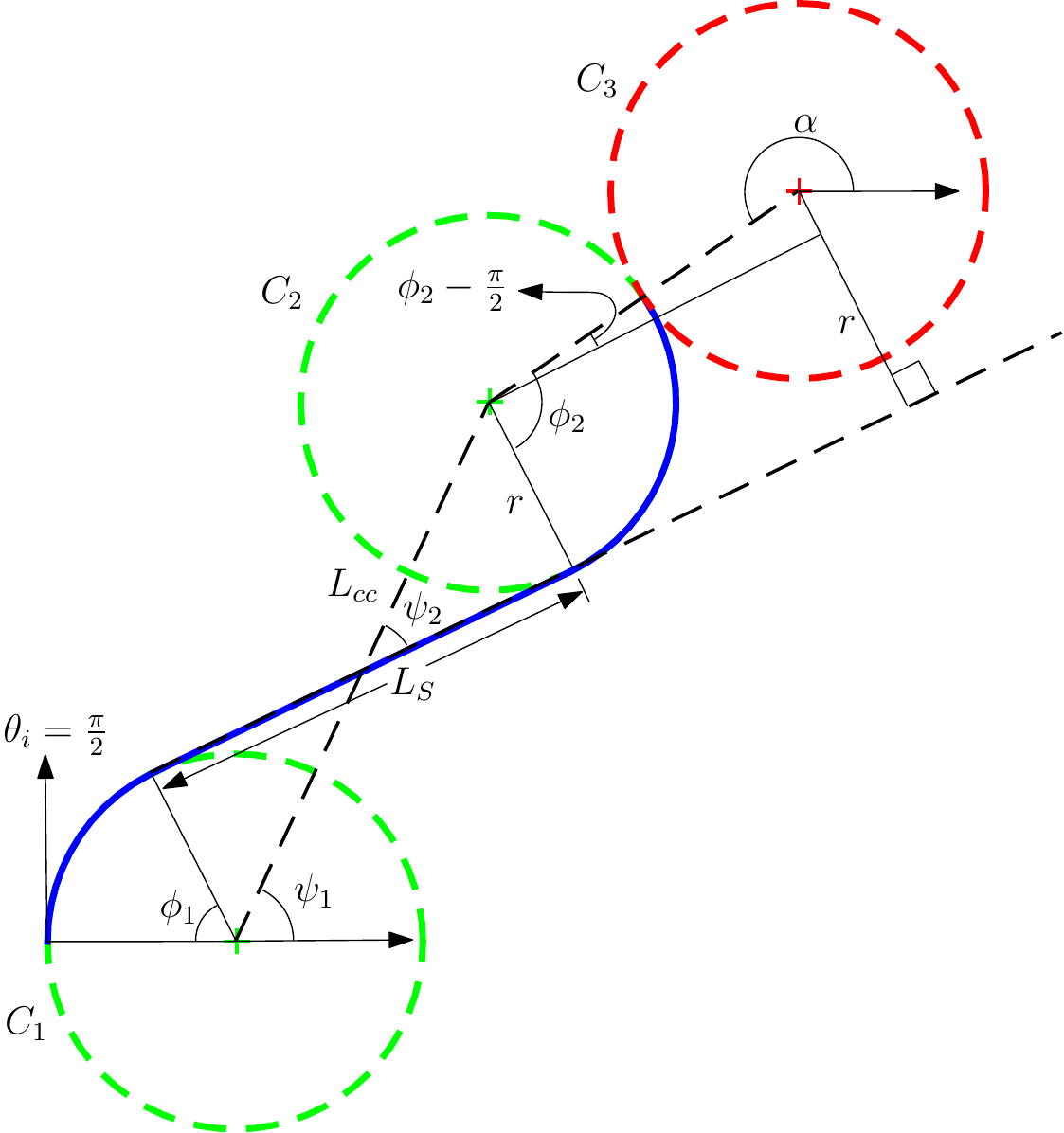}}
\end{center}
\caption{RSL Path}
\label{fig:pathrsl} 
\end{figure}
Without loss of generality we assume that the initial position is at the origin and the initial heading is towards the positive $y$-axis, i.e. the initial heading is at $\frac{\pi}{2}$ degrees with respect to $x$-axis as shown in the Fig. \ref{fig:pathrsl}. Let $(x,y)$ be the coordinates of the final position, which lies on the target circle, and $\theta$ be the final heading direction. The centre of the second circle $C_2$ is given by $(x-r\sin \theta, y+ r\cos \theta)$. The length of the $RSL$ path is the sum of the three segments $L_S+r(\phi_1+\phi_2)$. Let $L_{cc}$ be the distance between the centers of the circles $C_1$ and $C_2$, and it is given by 
\begin{equation}
L_{cc}=\sqrt{(x -r\sin \theta-r)^2+(y+ r\cos \theta)^2}. \label{eqn:rsldcc}
\end{equation}
The length of the straight line segments and the angles subtended by the arcs are given as follows:
\begin{flalign}
L_S &= \sqrt{L_{cc}^2 - 4r^2}, \label{eqn:rslls}\\
\phi_1 &= \mod( - \psi_1+\psi_2 + \frac{\pi}{2}, 2\pi),  \label{eqn:rslphi1} \\
\phi_2 &= \mod(\theta + \phi_1 - \frac{\pi}{2}, 2 \pi), \label{eqn:rslphi2}
\end{flalign}
where $\psi_1$ and $\psi_2$ are given as
\begin{flalign}
 \psi_1&= \mbox{atan2}\left(\frac{y+ r\cos \theta}{x -r\sin \theta-r}\right), \label{eq:psi1} \\
 \psi_2 &=\mbox{atan2}\left(\frac{2r}{L_S}\right). \label{eq:psi2}
\end{flalign}

By substituting  $(x,y) = (c+r\cos\alpha, d+r\sin\alpha )$ and $\theta  = \alpha - \frac{\pi}{2}$ in the eqs. (\ref{eqn:rsldcc}), (\ref{eq:psi1}) and  (\ref{eq:psi2}), $L_{cc}$, $\psi_1$ and $\psi_2$ reduces to the following:

\begin{flalign}
L_{cc} &= \sqrt{(c + 2r\cos \alpha-r)^2+(d + 2r\sin \alpha)^2} \label{eqn:rsldcccw} \\
\psi_1 &= \mbox{atan2}\left(\frac{d + 2r\sin \alpha}{c + 2r\cos \alpha-r}\right) \\
\psi_2 &= \mbox{atan2}\left(\frac{2r}{L_S}\right)
\end{flalign}

 \begin{lemma}\label{lem:rslmin}
 The extremum of the length of the RSL path with respect to the angular position on the target circle ($\alpha$) occurs when $\phi_2=\frac{\pi}{3}$ or $\frac{5 \pi}{3}$; and at these two positions, the direction of the straight line of RSL path (second segment) passes through the center of the target circle.
 \end{lemma}
\begin{proof}
Refer to the Appendix for the proof.
\end{proof}

\begin{corol}
The absolute value of the derivative of the length of the RSL path $\frac{d}{d \alpha} L_{RSL}$ represents the perpendicular distance from the center of the target circle to the straight line segment of the RSL path.
\end{corol}

\begin{proof}
From Fig. \ref{fig:pathrsl}, the perpendicular distance from the center of the target circle to the straight line segment of the RSL path is $r+2rsin(\phi_2-\frac{\pi}{2})$, which is equal to $r-2rcos(\phi_2)$. 
\end{proof}

\subsection{Counter clockwise tangential direction}\label{subsec:ccw}
\begin{prop} \label{prop:lslrslccw}
The length of the LSL and RSL paths vary linearly with the angular position ($\alpha$) on the target circle, except for a discontinuity when the final arc of the paths disappears and the paths degenerate to two segment paths (LS or RS), and this corresponds to the shortest LSL or RSL path.
\end{prop}
A sample plot of the length of these paths is shown in the Fig. \ref{fig:cscalccw}.  We will prove this proposition in the following subsections for the LSL and RSL paths.

\begin{figure}[h]
\begin{center}
{\includegraphics[width=2.75in]{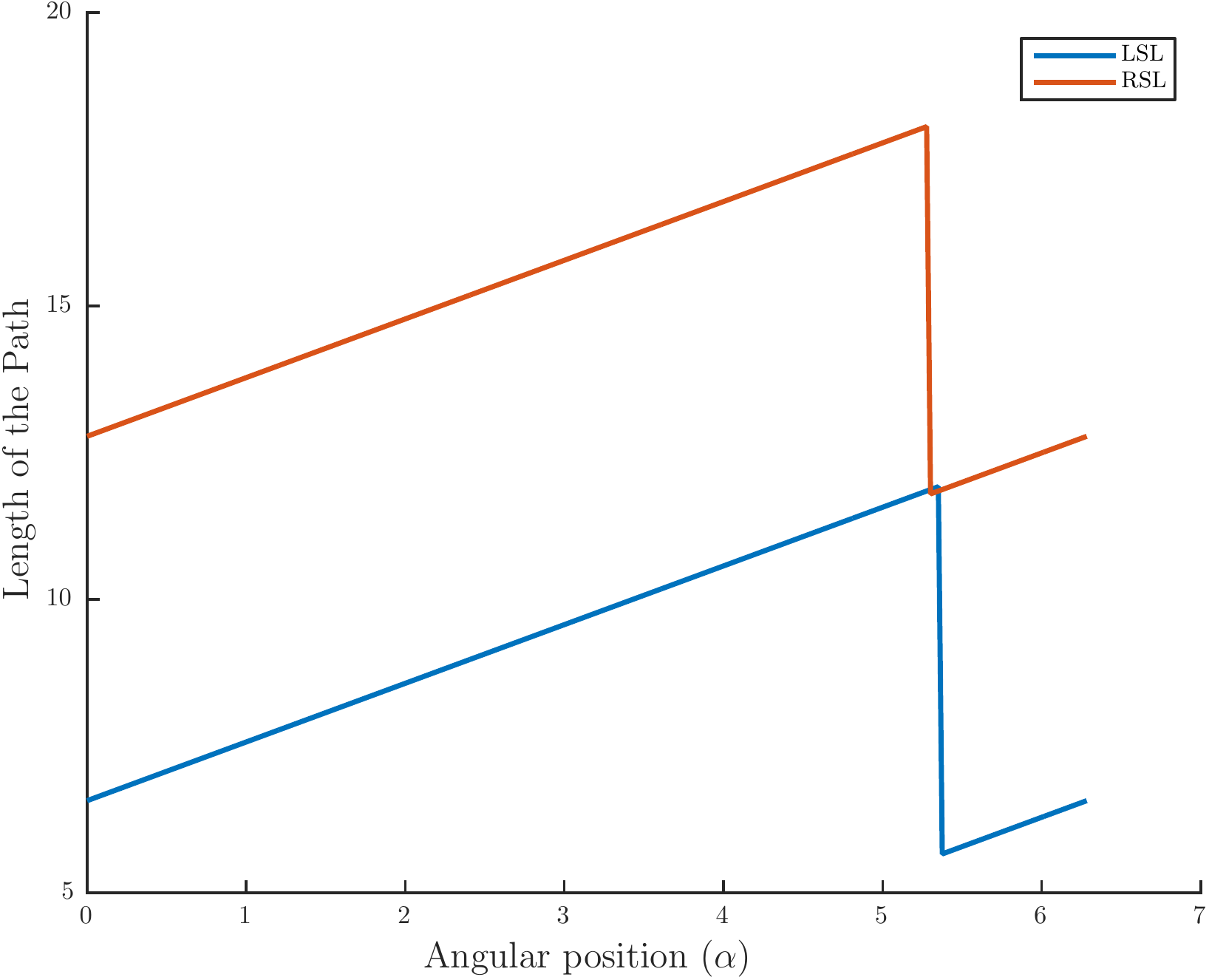}}
\end{center}
\caption{The length of the LSL and RSL paths vs the angular position on target circle. Final headings are counter-clockwise tangents to the target circle.}
\label{fig:cscalccw}
\end{figure}

\subsubsection{LSL Paths}

The final heading direction is counter-clockwise tangent to the target circle, and it implies that $\theta  = \alpha + \frac{\pi}{2}$. Substituting for $(x,y)$ and $\theta$ in equations (\ref{eqn:lslls} - \ref{eqn:lslphi1}) gives the following:

\begin{flalign}
L_S &= \sqrt{(c^2 +(d-r)^2}, \label{eqn:lsllsccw}\\
\phi_1 &= \mod\left(\mbox{atan2}\left(\frac{d-r}{c}\right),2\pi\right). \label{eqn:lslphi1ccw}
\end{flalign}

Clearly, the length of the first arc and the straight line segment are constant, and the second circle $C_2$ is co-located with the target circle $C_3$. From eq. (\ref{eqn:lslphi2}), one can see that length of the final arc changes linearly with $\alpha$, except for a discontinuity due to the modulus function. This confirms with the plot of the length of the LSL path as shown in the Fig. \ref{fig:cscalccw}. The minimum of the LSL path occurs when the length of the third segment goes to zero, or $\theta = \phi_1$, which means the direction of the straight line segment is same as the final heading. Thus, the LSL path is shortest at the position where the straight line segment is tangential to the target circle (point $B$ shown in the Fig. \ref{fig:lslccw}), and here the LSL path degenerates to a two segment path LS. For any other position, the LSL paths would consist of these two LS segments, and a third arc on target circle $C_3$ as shown in the Fig. \ref{fig:lslccw}.

\begin{figure}[htpb]
\begin{center}
{\includegraphics[width=3in]{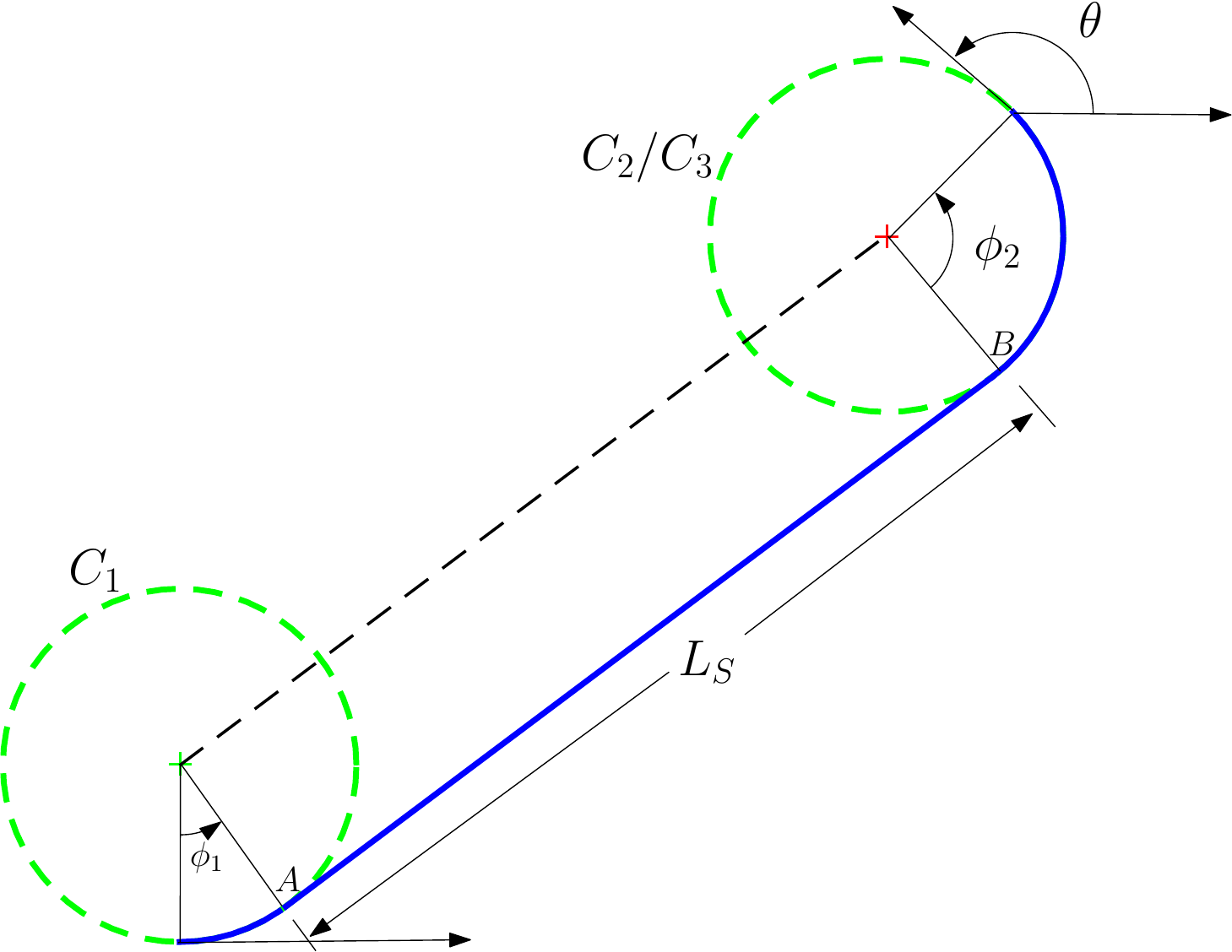}}
\end{center}
\caption{LSL Path where the final heading is a counter clockwise tangent to the target circle}
\label{fig:lslccw} 
\end{figure}

\subsubsection{RSL Paths}

In this case, as the final heading direction is in the counter-clockwise tangential direction, $\theta =\alpha + \frac{\pi}{2}$, and substituting this in eqs. (\ref{eqn:rsldcc})-(\ref{eqn:rslphi2}) gives the folllowing:
\begin{flalign}
L_{cc} &= \sqrt{(c-r)^2+(d)^2} \label{eqn:rsldccccw} \\
L_S &= \sqrt{L_{cc}^2 - 4r^2}, \label{eqn:rsllsccw}\\
\phi_1 &=\mod \left(\mbox{atan2}\left(\frac{2r}{L_S}\right) - \mbox{atan2}\left(\frac{d}{c-r}\right) + \frac{\pi}{2}, 2 \pi\right), \label{eqn:rslphi1ccw}\\
\phi_2 &= \mod\left(\theta + \phi_1 - \frac{\pi}{2}, 2 \pi\right). \label{eqn:rslphi2ccw}
\end{flalign}
Similar to the LSL path, the first two segments are constant for any $\alpha$, and the circles $C_2$ and $C_3$ are co-located. Hence, the length of the $RSL$ path varies only due to  the third segment. The value of $\phi_2$ is a piecewise linear function of $\alpha$, and the discontinuity as shown in the Fig. \ref{fig:cscalccw} is due to the modulus function which occurs when $\phi_2$ goes to zero. Also, the RSL path is shortest when this third segment goes to zero, \textit{i.e.} the point where the straight line segment is tangential to the target circle $C_3$ (point $B$ shown in the Fig. \ref{fig:rslccw}).
\begin{figure}[htpb]
\begin{center}
{\includegraphics[width=2.5in]{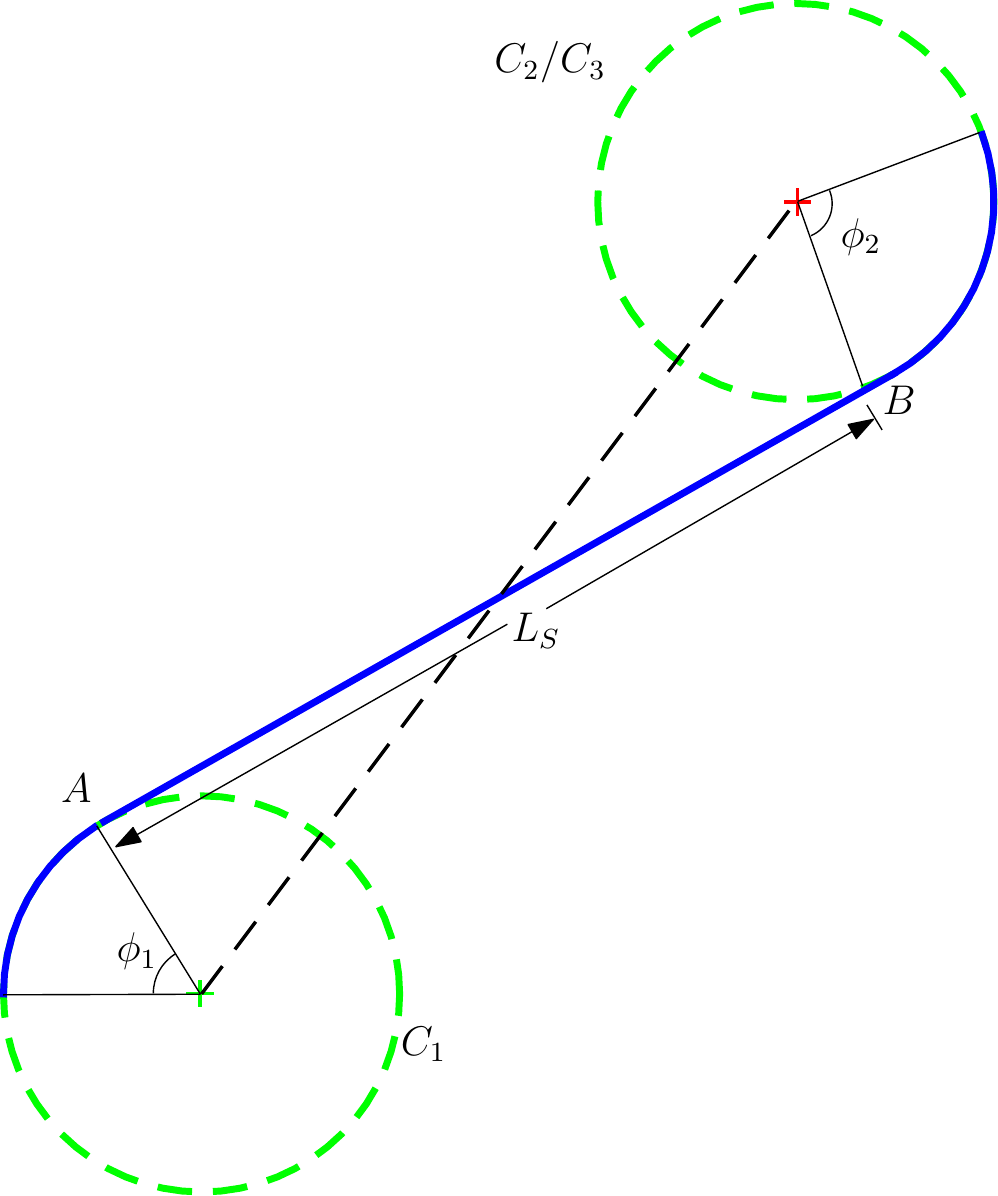}}
\end{center}
\caption{RSL Path with final heading in the counter-clockwise tangential direction to the target circle.}
\label{fig:rslccw} 
\end{figure}
Lemmas \ref{lem:lslmin} and \ref{lem:rslmin} completes the proof for Proposition \ref{thm:csccw}.

\section{Summary of results}

For a given rotational direction of the tangent at the target circle, using the analysis in Section \ref{sec:main}, we could find the shortest LSL and RSL paths. Due to the symmetry of the CSC paths, one could evidently extend these conditions to the RSR and LSR paths. In summary, one could find the shortest of any CSC path based on whether the rotational direction of the second arc ($C_2$) of the CSC path and the direction of the tangent at the target circle are same or different. $(i)$  If they are same, the circles $C_2$ and $C_3$ will be co-located, and the shortest CSC path occurs when third arc disappears and the path degenerates to a CS path. $(ii)$ If they are different, the shortest path occurs at the angular position on the target circle at which the second arc will be of angle $\frac{\pi}{3}$, and the direction of the straight line passes through the center of the target circle. Now the path of minimum length among the shortest LSL, LSR, RSR and RSL gives the shortest CSC path: $CSC_{min} = \min \{LSL_{min}, LSR_{min}, RSR_{min}, RSL_{min} \}$. We summarize the conditions for the minimum  and the discontinuity of the four CSC paths in the Tables \ref{tab:ressum1} and \ref{tab:ressum2}. Though we assume the initial heading to be zero for LSL and $\frac{\pi}{2}$ for the RSL paths in the Section \ref{sec:main}, the results could be generalized for any initial heading. We list the conditions in the Tables \ref{tab:ressum1} and \ref{tab:ressum2} for any given initial heading $\theta_i$.


\begin{table}[t]
\caption{SUMMARY OF RESULTS FOR COUNTER-CLOCKWISE TANGENT AT TARGET CIRCLE}
\begin{center}
\label{tab:ressum1}
\begin{tabular}{ccl}
& & \\ 
\hline
Path type & Minimum & Discontinuity \\
\hline
LSL & $\alpha = \theta_i+\phi_1 - \frac{\pi}{2} $ & $\alpha = \theta_i+\phi_1 - \frac{\pi}{2} $\\
RSL & $\alpha = \theta_i -\phi_1 - \frac{\pi}{2} $ &  $\alpha =\theta_i -\phi_1 - \frac{\pi}{2}$ \\
RSR & $\alpha = \theta_i-\phi_1 - \frac{5\pi}{6} $ & $ (i) \alpha =\theta_i -\phi_1 - \frac{\pi}{2}$ \\
& & $(ii) (\phi_1=0)$ \\
LSR & $\alpha = \theta_i + \phi_1 - \frac{5\pi}{6} $ & $(i)\alpha = \theta_i+\phi_1 - \frac{\pi}{2}$ \\
& & $(ii) (\phi_1=0)$ \\
\hline
\end{tabular}
\end{center}
\end{table}

\begin{table}[t]
\caption{SUMMARY OF RESULTS FOR CLOCKWISE TANGENT AT TARGET CIRCLE}
\begin{center}
\label{tab:ressum2}
\begin{tabular}{ccl}
& & \\ 
\hline
Path type & Minimum & Discontinuity \\
\hline
LSL & $\alpha = \theta_i+\phi_1 + \frac{5\pi}{6} $  & $(i) \alpha = \theta_i+\phi_1 + \frac{\pi}{2} $\\
& & $(ii) (\phi_1=0)$ \\
RSL & $\alpha = \theta_i -\phi_1 + \frac{5\pi}{6} $ &  $(i) \alpha =\theta_i -\phi_1 + \frac{\pi}{2}$ \\
& & $(ii) \phi_1=0$\\
RSR & $\alpha = \theta_i -\phi_1 + \frac{\pi}{2} $ &  $\alpha = \theta_i -\phi_1 + \frac{\pi}{2} $ \\
LSR & $\alpha = \theta_i + \phi_1 + \frac{\pi}{2} $ & $\alpha = \theta_i + \phi_1 + \frac{\pi}{2} $\\
\hline
\end{tabular}
\end{center}
\end{table}


\section{Conclusion}
We considered a generalization of the Dubins path problem, which is to find a shortest Dubins path from an initial configuration to a final position that lies on a given target circle and the final heading is a tangent to the circle. We assume the distance between the initial and final configurations is always greater than four times minimum turn radius, and this leads to only paths of type CSC. We characterized the length of the four CSC paths with respect to the angular position on the target circle, and presented the necessary conditions for the minimum and maximum. The minimum of these four shortest paths would give the shortest Dubins path to the circle. We also derive for the angular positions ($\alpha$), at which the lengths of these paths are discontinuous. Here, we assume the minimum turn radius for the paths and the radius of the target circle are equal, a first future direction would be to extend these results to the case when these two radii are different. Another future direction is to characterize the CCC paths, and all the six classes of paths would complete the analysis for the shortest Dubins path to circle.

\bibliographystyle{asmems4}
\bibliography{DubinsToCircle}

\section{Appendix A}

\begin{figure}
\begin{center}
\subfigure[Minimum LSL path ]{\includegraphics[width=2.25in]{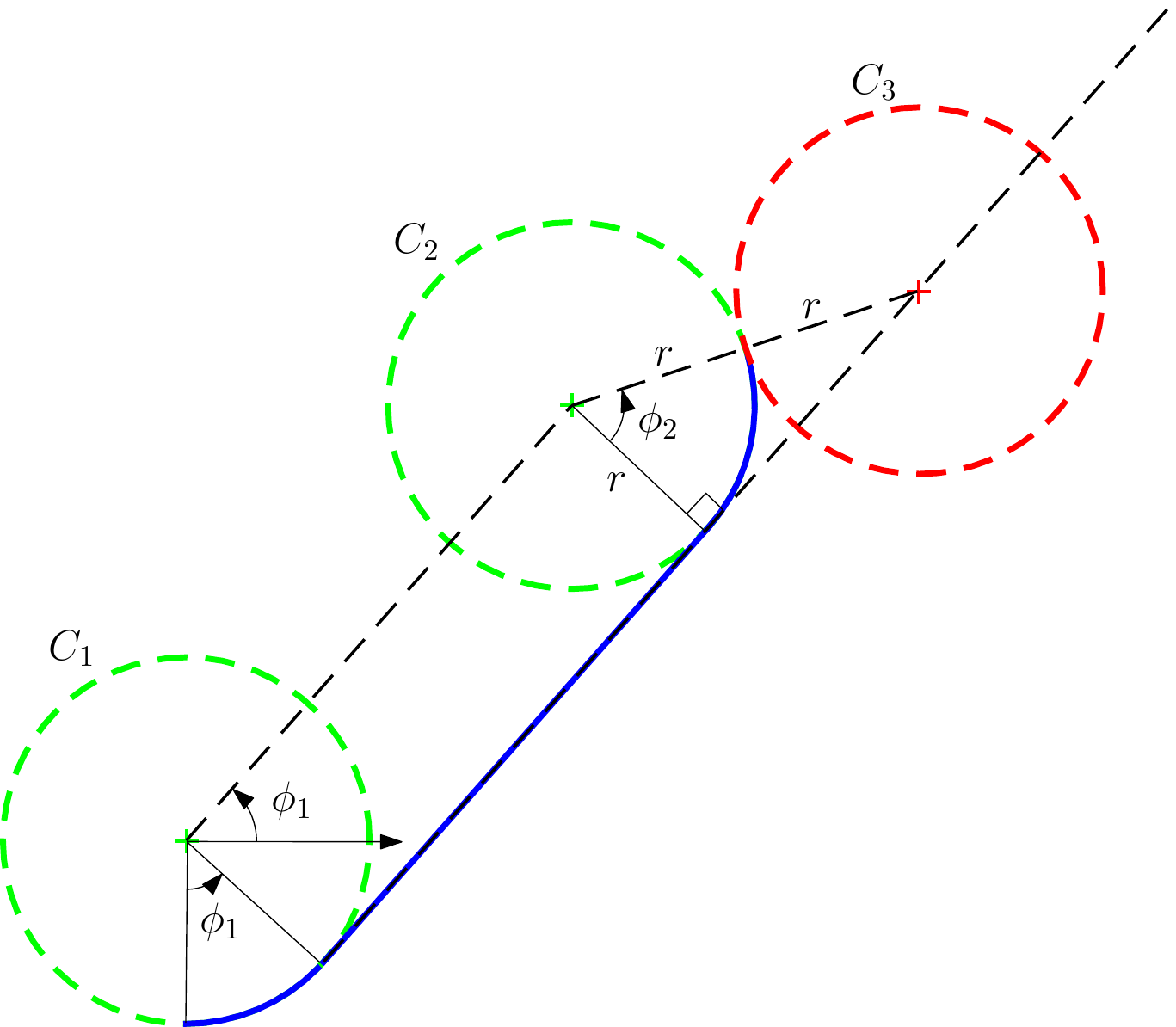}\label{fig:lslmin}}
\subfigure[Maximum LSL path]{\includegraphics[width=2.25in]{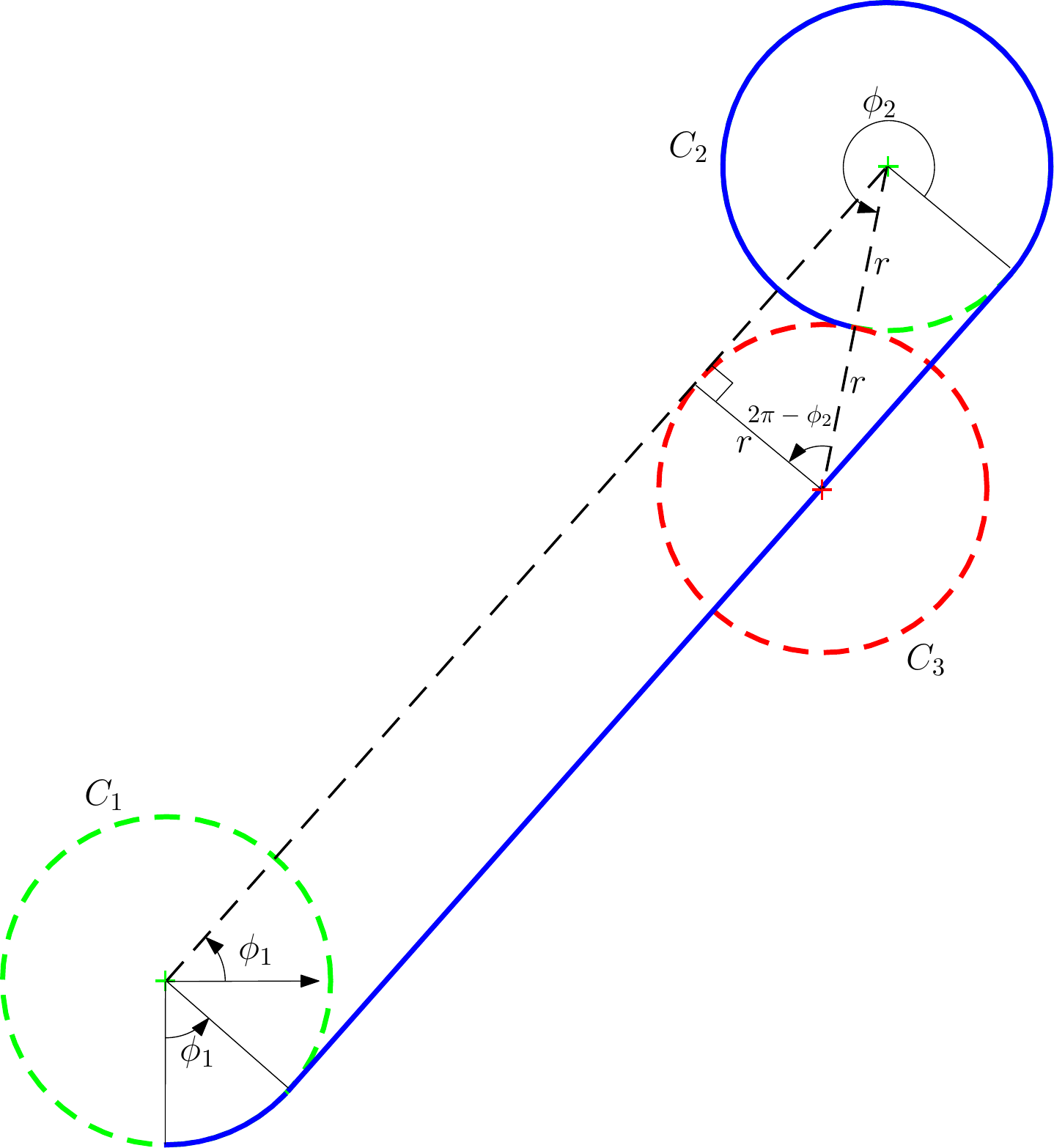}\label{fig:lslmax}}
\end{center}
\caption{Minimum and maximum LSL paths, shows that the direction of straight line segment passes through the centre of the target circle, when final arc is $\frac{\pi}{3}$ or $\frac{5\pi}{3}$.}
\label{fig:lslminmax} 
\end{figure}

\subsection{Proof of Lemma \ref{lem:lslmin}} \label{sec:lslminproof}
\begin{proof}
The length of the LSL path to the target circle is given by the summation of the three segments of the path, $$L_{LSL}=L_S+r(\phi_1+\phi_2)$$. The length of the LSL path, $L_{LSL}$ is piecewise continuous due to the modulus operator in the values of $\phi_1$ and $\phi_2$. And therefore, the derivative of this function is not defined when $\mbox{atan2}(\frac{d+2r\sin \alpha-r}{c+2r\cos \alpha}) = 0$ and $\theta-\phi_1 = 0$. To find the extremum of $L_{LSL}$, we will differentiate it with respect to $\alpha$ and make it equal to zero. At the points where the derivative exists, it is given as following:

\begin{flalign*}
\frac{d}{d \alpha} L_{LSL} &= \frac{d}{d \alpha} \left( L_S  + r(\alpha - \frac{\pi}{2}) \right)\\
&= -2r \sin \alpha \left(\frac{c+2r \cos \alpha}{L_S} \right)+ 2r \left(\frac{d+2r \sin \alpha-r}{L_S}\right)  + r \\
&= 2r(-\sin \alpha \cos \phi_1 + \cos \alpha \sin \phi1) + r \\
&= 2r \sin(\phi_1-\alpha)+r \\
&=2r \sin(\phi_2 - \frac{\pi}{2}) + r\\
&=-2r\cos(\phi_2) + r
\end{flalign*}
Now equating the first derivative gives the necessary conditions for the extremum.
\begin{flalign*}
& \cos(\phi_2) = \frac{1}{2} \\
\implies & \phi_2 = \frac{\pi}{3} \mbox{ or } \frac{5 \pi}{3}
\end{flalign*}
It is clear from the Fig. \ref{fig:lslminmax} that when $\phi_2=\frac{\pi}{3}$ or $\phi_2=\frac{5\pi}{3}$, the direction of the straight line segment of the LSL path passes through center of the target circle.
\end{proof}

\subsection{Proof of Lemma \ref{lem:rslmin}}\label{sec:rslminproof}
\begin{proof}
The length of the RSL path to the target circle is given by the summation of the three segments of the path, $$L_{RSL}=L_S+r(\phi_1+\phi_2)$$. Similar to the LSL path, The length of the RSL path, $L_{RSL}$ is piecewise continuous due to the modulus operator in the values of $\phi_1$ and $\phi_2$. Hence, its derivative is not defined for the $\alpha$ values where the following satisfies:
\begin{flalign}
& \mbox{atan2}\left( \frac{2r}{L_S} \right) - \mbox{atan2} \left( \frac{d+2r\sin \alpha}{c+2r\cos \alpha} \right) +\frac{\pi}{2}= 0, \\
& \theta + \phi_1 -\frac{\pi}{2}= 0.
\end{flalign}
To find the extremum of $L_{RSL}$, we will differentiate it with respect to $\alpha$ and make it equal to zero. At the points where the derivative exists, it is given as following:
\begin{flalign*}
\frac{d}{d \alpha} L_{RSL} &= \frac{d}{d \alpha} \left( L_S  + r(2 \phi_1 + \alpha - \pi) \right)\\
&=  \frac{d}{d \alpha} \left( L_S  + r( -2\psi_1 + 2\psi_2 + \alpha) \right)
\end{flalign*}

First, we wil derive the first derivative of these individual terms:
\begin{flalign*}
\frac{d}{d \alpha} L_{S} &= 4r \left(cos \alpha \frac{d+2r \sin\alpha}{L_S} -\sin \alpha \frac{c+2r \cos \alpha -r}{L_S} \right) \\
&= 4r \left(\cos \alpha \frac{d+2r \sin\alpha}{L_{cc}}\frac{L_{cc}}{L_S} -\sin \alpha \frac{c+2r \cos \alpha -r}{L_{cc}}\frac{L_{cc}}{L_S} \right) \\
&=\frac{2r}{\cos \psi_2} \left( \cos \alpha \sin \psi_1 -\sin \alpha\cos \psi_1  \right) \\
&= 2r \frac{sin (\psi_1-\alpha)}{\sin (\phi_1+\psi_	1)}
\end{flalign*}

\begin{flalign*}
\frac{d}{d \alpha} \psi_1 &= \frac{2r}{d^2_{cc}} \left( \sin \alpha (d+2r \sin\alpha) + \cos \alpha (c+2r \cos\alpha-r)  \right) \\
&= \frac{2r}{L_{cc}} \left( \sin \alpha \sin \psi_1 + \cos \alpha \cos \psi_1 \right) \\
&= \sin(\psi_2) \cos(\alpha-\psi_1)\\
&= -\cos(\psi_1+\phi_1) \cos(\alpha-\psi_1)
\end{flalign*}

\begin{flalign*}
\frac{d}{d \alpha} \psi_2 &= \frac{-4r^2}{d^2_{cc}} \left( \cos \alpha \frac{d+2r \sin\alpha}{L_S} - \sin \alpha \frac{c+2r \cos\alpha+r}{L_S}  \right) \\
&= \frac{-4r}{L_{cc}} \frac{r}{L_S} \left( \cos \alpha \sin \psi_1 - \sin \alpha \cos \psi_1 \right)  \\
&= \sin(\psi_2) \tan \psi_2 \sin(\psi_1-\alpha) \\
&= \frac{\cos^2(\psi_1+\phi_1)}{\sin(\psi_1+\phi_1)} \sin(\alpha-\psi_1)
\end{flalign*}

\begin{flalign*}
\frac{d}{d \alpha} L_{RSL} &= \frac{2r \sin(\psi_1-\alpha)}{\sin (\phi_1 + \psi_1)} +\frac{2r\cos^2(\phi_1+\psi_1) \sin(\alpha-\psi_1)}{\sin (\phi_1 + \psi_1)} \\
&\quad+2r \cos(\phi_1+\psi_1) \cos(\alpha-\psi_1)+r\\
&=2r\cos(\alpha + \phi_1)+r \\
&=2r\cos(\phi_2+\pi) + r\\
&=-2r\cos(\phi_2) + r
\end{flalign*}

The maximum or minimum occurs when $\frac{d}{d \alpha} L_{RSL}=0$, which implies that
\begin{flalign*}
& \cos(\phi_2) = \frac{1}{2} \\
\implies &\phi_2 = \frac{\pi}{3} or \frac{5\pi}{3}.
\end{flalign*}
Clearly, when the second arc is of angle $\frac{\pi}{3}$ or $\frac{5\pi}{3}$, the direction of straight line segment passes through center of the target circle.
\end{proof}

\end{document}